\documentclass[12pt,a4paper,twoside,reqno]{amsart}

\usepackage{tikz}
\usepackage{amsmath}
\usepackage{amssymb,amsfonts,mathrsfs}
\usepackage[colorlinks=true,linkcolor=blue,citecolor=blue]{hyperref}
\usepackage[driver=pdftex,margin=3cm,heightrounded=true,centering]{geometry}

\usepackage{math60,mlthm}
\usepackage{local}

\numberwithin{equation}{section}

\begin{document}

\title[Multi-parameter resolvent trace expansions]
{Multi-parameter resolvent trace expansion \\ I. for elliptic boundary problems}

\author{Boris Vertman}
\address{University M\"unster,
Einsteinstra\ss e 62, 
48149 M\"unster,
Germany}
\email{vertman@uni-muenster.de}

\date{This document was compiled on: \today}
\thanks{The author was supported by the 
        Bonn Hausdorff Center and University M\"unster.}

\begin{abstract}
Various aspects of semi-classical analysis require discussion of 
asymptotic expansions in two parameters weighted with smooth functions,
which in the setting of manifolds with boundary does not follow
in an obvious way from the classical parametric elliptic theory. 
In the present paper we establish multi-parameter resolvent trace expansions for 
elliptic boundary value problems, polyhomogeneous both in the resolvent 
and the auxiliary parameters. Such multi-parameter resolvent trace expansions have been used in the
setting of revolution surfaces with unexpected applications to regularized
sums of zeta-determinants. Another example where these ideas play a 
prominent role is the asymptotics of determinants for discrete Laplacians on
tori under growing discretization parameter. 
\end{abstract}

\maketitle
\tableofcontents

\ \\[-10mm]
%%%%%%%%%%%%%%%%%%%%%%%%%%%%
\section{Introduction and formulation of the main result}\label{intro}
%%%%%%%%%%%%%%%%%%%%%%%%%%%%

\subsection{Introduction and motivation} Interest in the multi-parameter resolvent trace expansions arises in view 
of their various applications. A particular example of such an application is a joint
project with Lesch \cite{LesVer:RSD} where we compare the zeta-determinant
of the Laplace Beltrami operator on a surface of revolution with a regularized
sum of zeta-determinants for scalar operators arising from a spectral decomposition 
on the cross section of the revolution surface. 
\medskip

More generally, consider a manifold with fibered
boundary $\mathbb{S}^d \times M$, with an elliptic boundary problem $(\mathscr{A}, \mathscr{B})$ represented 
after the eigenspace decomposition on $\mathbb{S}^d$ by an infinite sum of 
elliptic boundary problems $(A_n,B_n), n\in \N_0$ on $M$. Assume the boundary problems 
admit well defined zeta-determinants. Then the methods elaborated 
in \cite{LesVer:RSD} together with our main theorem here equate
the zeta-determinant of $\mathscr{A}_{\mathscr{B}}$ to the regularized 
sum of zeta determinants for $A_{n,B_n}$, up to a locally computable error term.
\medskip

Another application of multi-parameter resolvent trace expansions appears in the 
study of asymptotics for determinants of graph Laplacians on discretized tori. In 
\cite{Ver} we show that the constant term in the asymptotics of the combinatorial 
determinant for the graph Laplacian on the discretized tori is given in terms of the 
zeta-determinant of the Laplace Beltrami operator on the smooth torus manifold.
\medskip

Finally, we expect a multi-parameter resolvent trace expansion 
to play a role in establishing an asymptotic expansion of the Bergmann kernel in the 
setting of complex manifolds with singular metric structure along divisors.

\subsection{Formulation of the main result} Consider a compact manifold $M$ of dimension $m$ with boundary $\partial M$,
equipped with a Hermitian vector bundle $(E, h^E)$ of rank $p$. We recall the notion of an elliptic 
boundary problem from Seeley \cite{See:TRO}. Let $A\in \textup{Diff}^{\, q}(M, E)$ 
denote a differential operator on $M$ with values in $E$ of $q$-th order with $pq\in 2\N_0$.
The operator $A$ is elliptic if its principal symbol $\sigma(A) (p,\zeta)$ is invertible 
for $(p,\zeta) \in T^*M\backslash \{0\}$. Assume, $A$ satisfies the Agmon condition 
in a fixed cone $(-\Gamma')$ with $\Gamma'=\{z\in \C \mid \arg(z) \in (\theta_1, \theta_2)\}$ of the complex plane, i.e. 
$(\sigma(A) +z^q)$ is invertible for $z^q \in \Gamma'$. Consider a system of differential 
operators $B=(B_1,..,B_{pq/2})$ on $\partial M$, such that $(A,B)$ defines an elliptic 
boundary problem satisfying the Agmon condition on $\Gamma'$ in the sense of \cite[Def. 1,2]{See:TRO}.
\medskip

Under this setup, $(A,B)$ defines a closed unbounded operator $A_B$ on $L^2(M,E)$, obtained 
as the graph closure of $A$ acting on $u\in C^\infty(M,E)$ satisfying the boundary conditions $Bu=0$.
Moreover, $(A_B + z^q)$ is invertible for $z^q\in \Gamma'$ sufficiently large, cf. \cite[p. 911]{See:TRO}, and 
the seminal work of Seeley \cite{See:TRO} establishes an expansion of the resolvent 
$(A_B + z^q)^{-1}$ as $|z|\to \infty$. In the present paper we fix a finite collection 
of scalar smooth potentials $V_1,..,V_n\in C^\infty(M)$ which are assumed to be nowhere vanishing 
along the boundary $\partial M$. Set $\Gamma:= \{z \in \C \mid z^q\in \Gamma'\}$ and write $\Gamma^{\,n}$ for 
its $n$-th Cartesian product. Consider $\lambda = (\lambda_1, ... ,\lambda_n) \in \Gamma^{\,n}$ and the 
corresponding multi-parameter family 
\[
A_B(\lambda) := A_B + \sum_{k=1}^n \lambda_k^q V_k.
\]
For $z\in \Gamma$ sufficiently large, $(A_B(\lambda)+ z^q)$ is invertible. 
For $qN >m$ the $N$-th power of the resolvent $(A_B(\lambda)+ z^q)^{-N}$ 
is trace class and our main result establishes a multi-parameter expansion 
of the resolvent trace $\textup{Tr}(A_B(\lambda)+ z^q)^{-N}$, 
polyhomogeneous in $(z, \lambda)\in \Gamma^{\,n+1}$.

\begin{theorem}\label{phg-trace}
Consider any multiindex $\A \in \N_0^n$ and $\beta \in \N_0$. Fix $N\in \N$ such that 
$qN >m$. Then there exist  $e_i \in C^\infty (M \times (\Gamma^{n+1} \cap \mathbb{S}^n))$,
such that
$$\partial^\A_\lambda \partial_z^\beta \Tr (A_B(\lambda) + z^q)^{-N}
\sim \sum_{j=0}^\infty e_j \left(\frac{(\lambda, z)}{|(\lambda, z)|} \right)
|(\lambda, z)|^{-1-qN - j -|\A| -\beta +m}.$$
\end{theorem}

Let us mention that the statement has a straightforward extension to the case
of additional summands in the expression for $A_B(\lambda)$ of the form $\lambda^\alpha_k D_k$,
where $D_k\in \textup{Diff}^{\, p}(M, E)$ and $\alpha + p <q$. Moreover the arguments extend 
to the case of matrix-valued potentials $V_1,..,V_n\in C^\infty(M,\textup{End}(E))$, assuming that the action of their 
restrictions to the boundary is scalar.\medskip

This paper is organized as follows. 
In \S \ref{blowup-sec} we recast the symbolic expansion of the
resolvent $(A_B+z^q)^{-1}$ in terms of polyhomogeneity properties 
of its Schwartz kernel lifted to an appropriate blowup of $\R^+ \times M^2$.
In \S \ref{comp-sec} we establish a composition result for the polyhomogeneous
conormal distributions on the blowup of $\R^+ \times M^2$. 
In \S \ref{multi-sec} we employ this microlocal characterization of the 
resolvent kernel to establish the multi-parameter resolvent trace expansion.
While the resolvent exists by standard arguments, its microlocal description
is obtained by constructing the corresponding Schwartz kernel as a polyhomogeneous conormal distribution
on the blowup space using an iterative Neumann series argument

%%%%%%%%%%%%%%%%%%%%%%%
\section{Resolvent kernel as a polyhomogeneous conormal distribution} \label{blowup-sec}
%%%%%%%%%%%%%%%%%%%%%%%%%

\subsection{Resolvent of an elliptic boundary value problem}
Seeley \cite{See:TRO} provides a careful construction of the resolvent 
for the elliptic boundary problem $A_B$. More precisely, in view of 
\cite[Theorem 1, Lemma 2, (25), (32), (49))]{See:TRO} we may state the following 

\begin{theorem}\label{thm-seeley}
Let $\Gamma_R:= \{\mu\in \Gamma \mid |\mu|\geq R\}$. Consider a local coordinate
neighborhood $\mathscr{U}\subset [0,1)_x \times \partial M$ in the collar of the boundary,
with local coordinates $\{x,y=(y_1,..,y_{m-1})\}$. Then for any $j\in \N_0$ there exist 
$d_{-q-j} \in C^\infty (\mathscr{U} \times \R^m \times \Gamma, \textup{Hom}(E\restriction \mathscr{U}))$,
homogeneous of order $(-q-j)$ in $(x^{-1}, \zeta, \gamma, \mu)$, where 
$(x, y, \zeta, \gamma, \mu)\in \mathscr{U} \times \R^{m-1}_\zeta\times \R_\gamma \times \Gamma_\mu$, 
such that for $R>0$ sufficiently large and $\mu\in \Gamma_R$ the Schwartz kernel\footnote{The notation 
Op denotes the integral kernel of the operator rather than the operator itself.} 
\begin{align*}
&(A_B+\mu^q)^{-1}(x,y,\wx,\wy) - (A+\mu^q)^{-1}(x,y,\wx,\wy) \\
&- \sum_{j=0}^{N-1} (2\pi)^{-m} \int_{\R^{m-1}} \int_\R e^{i \langle y-\wy , \zeta \rangle} e^{-i \wx \gamma }
d_{-q-j}(x,y,\zeta,\gamma,\mu) d\gamma d\zeta\\
&=: K_{A_B}(x,y,\wx,\wy;\mu) - K_{A}(x,y,\wx,\wy;\mu) - \sum_{j=0}^{N-1} \textup{Op}(d_{-q-j}),
\end{align*}
is uniformly $O(|\mu|^{m-q-N})$ as $|\mu|\to \infty$. Here, the first term $K_{A_B}$ denotes the 
resolvent kernel of $A_B$ near the boundary, whereas the second term $K_A$ refers to an interior 
symbolic parametrix\footnote{The interior symbolic parametrix $K_A$ is obtained by inverting the 
symbol of $(A+\mu^q)$ algebraically by an iteration as in \cite[(5), (6)]{See:TRO}, and studying the corresponding operator. Such a parametrix is defined
independently of the boundary conditions and provides a good parametrix for $(A+\mu^q)^{-1}$
away from boundary. } defined without taking into account the boundary conditions. Note however that
the symbol corresponding to $K_A$ determines the choice of $d_{-q-j}, j\in \N_0$ by a system of differential 
equations \cite[(9),(10),(11)]{See:TRO}.
\end{theorem}

\subsection{Construction of the resolvent blowup space}
Assume for simplicity that $\mu$ varies along a ray within $\Gamma$, which we identify with $\R^+$.
Then the Schwartz kernel $K_{A_B}$ of the resolvent $(A_B +\mu^q)^{-1}$ is a distribution on $\R^+_{1/\mu} \times M^2$.
Choose local coordinates $(x,y)$ and $(\wx,\wy)$ on the two copies of $M$ in a collar neighborhood 
of the boundary, where $x$ and $\wx$ are the boundary defining functions. The Schwartz kernel $K_{A_B}$
admits a non-uniform behaviour at 
\begin{equation}
\label{D-C}
\begin{split}
&\mathscr{C}:=\{\mu=\infty, x=\wx=0, y=\wy\}, \\
&\mathscr{D}:=\{\mu=\infty, (x,y) = (\wx,\wy)\}. 
\end{split}
\end{equation}
This non-uniform 
behaviour is resolved by considering an appropriate blowup $\mathscr{M}^2_b$ of $\R^+\times M^2$
at $\mathscr{C}$ and $\mathscr{D}$, a procedure introduced by Melrose, see \cite{Mel:TAP}, 
such that the kernels $K_{A_B}, K_A$ lift to polyhomogeneous distributions on the manifold with corners 
$\mathscr{M}^2_b$ in the sense of the following definition.

\begin{defn}\label{phg}
Let $X$ be a manifold with corners, with all boundary faces embedded, and $\{(H_i,\rho_i)\}_{i=1}^N$ an enumeration 
of its boundaries and corresponding defining functions. For any multi-index $b= (b_1,
\ldots, b_N)\in \C^N$ we write $\rho^b = \rho_1^{b_1} \ldots \rho_N^{b_N}$.  Denote by $\mathcal{V}_b(X)$ the space
of smooth vector fields on $X$ which are
tangent to all boundary faces. All distributions on $X$ will be assumed to be restrictions of distributions 
defined on a closed manifold extending $X$. A distribution $\w$ on $X$ is said to be
conormal if $\w\in \rho^b L^\infty(X)$ for some $b\in \C^N$, and 
$V_1 \ldots V_\ell \w \in \rho^b L^\infty(X)$, for all $V_j \in \mathcal{V}_b(X)$ 
and for every $\ell \geq 0$. An index set 
$E_i = \{(\gamma,p)\} \subset {\mathbb C} \times {\mathbb N}$ 
satisfies the following hypotheses:
\begin{enumerate}
\item if $(\gamma_j,p_j) \in E_i$ and $|(\gamma_j,p_j)| \to \infty$ then $\textup{Re}(\gamma_j) \to \infty$,
\item if $(\gamma,p) \in E_i$, then $(\gamma+j,p') \in E_i$ for all $j \in \N_0$ and $0 \leq p' \leq p$,
\end{enumerate}
An index family $E = (E_1, \ldots, E_N)$ is an $N$-tuple of index sets. 
Finally, we say that a conormal distribution $w$ is polyhomogeneous on $X$ 
with index family $E$, we write $\w\in \mathscr{A}_{\textup{phg}}^E(X)$, 
if $\w$ is conormal and if in addition, near each $H_i$, 
\begin{align}\label{A}
\w \sim \sum_{(\gamma,p) \in E_i} a_{\gamma,p} \rho_i^{\gamma} (\log \rho_i)^p, \ 
\textup{as} \ \rho_i\to 0,
\end{align}
with coefficients $a_{\gamma,p}$ conormal on $H_i$, polyhomogeneous with index $E_j$
at any $H_i\cap H_j$\footnote{The asymptotic expansion for $\w$ means in precise terms
that the difference between $\w$ and any finite portion of the expansion vanishes at the
rate of the next term in the expansion, with a corresponding property for all higher
derivatives, and with decay order at other boundary faces determined by the index 
family $E$.}. 
\end{defn}

We also need to consider polyhomogeneous distributions on a manifold with corners $X$, 
conormal to an embedded submanifold $Y\subset X$. 
The basic space $I^m(\R^n, \{0\})$ consists of compactly supported 
distributions whose Fourier transform is given by a symbol of order $(m-n/4)$. 
$I^m(\R^n, \{0\})$ is invariant under
local diffeomorphisms and thus makes sense on any manifold around an isolated point.
\medskip

For an embedded $k$-submanifold $S\subset X$, any point in $S$ admits 
an open neighborhood $\mathscr{V}$ in $X$ which can be 
locally decomposed as a product $\mathscr{V}=X' \times X''$
so that $\mathscr{V} \cap S= X' \times \{p\}, p \in X''$. The
space $I^m(X,S)$ is defined (locally) as the space of smooth functions on $X'$ 
with values $I^{m+\dim X' /4}(X'',\{p\})$.
The normalization is chosen to give pseudo-differential
operators their expected orders. All distributions in $I^m(X,S)$ are locally restrictions of distributions
on an ambient space, which are conormal to any smooth extension of $S$ across $\partial X$.
\medskip

Choosing now index sets $E$ for each boundary face of $X$ as in Definition 
\ref{phg}, we define a space $\mathscr{A}_{\textup{phg}}^E(X,Y)$
as the space of distributions conormal to $Y$, with
polyhomogeneous expansions as in \Eqref{A} at all boundary faces and 
with coefficients conormal to the intersection of $Y$ with each boundary face. 
\medskip

We now continue with the definition of a blown-up space\footnote{Do not confuse 
$\mathscr{M}^2_b$ with the b-streched double space in \cite{Mel:TAP}} $\mathscr{M}^2_b$, to which 
the Schwartz kernels $K_A$ and $\textup{Op}(d_{-q-j})$ lift to polyhomogeneous distributions,
possibly conormal to an embedded submanifold. Blowing up $\R^+\times M^2$ at $\mathscr{C}$ and $\mathscr{D}$
amounts in principle to introducing polar coordinates in $\R^+\times (\R^+)^2$ 
at 
\begin{equation}
\label{D-C-model}
\begin{split}
&\mathscr{C}':=\{(1/ \mu, x,y) \in \R^+\times (\R^+)^2 \mid \mu=\infty, x=\wx=0, y=\wy\}, \\
&\mathscr{D}':=\{(1/ \mu, x,y) \in \R^+\times (\R^+)^2 \mid \mu=\infty, (x,y) = (\wx,\wy)\},
\end{split}
\end{equation} 
together with a unique minimal differential 
structure with respect to which these coordinates are smooth. \medskip

We first perform a blowup of $\mathscr{C}$.
The resulting space $[\R^+\times M^2; \mathscr{C}]$ is defined as the union of
$\R^+\times M^2 \setminus \mathscr{C}$ with the inward-pointing spherical normal bundle of 
$\mathscr{C}$ in $\R^+\times M^2$. The blowup $[\R^+\times M^2; \mathscr{C}]$ 
is endowed with the unique minimal differential structure with respect to which smooth 
functions in the interior of $\R^+\times M^2$ and polar coordinates on $\R^+\times M^2$ 
around $\mathscr{C}$ are smooth. This blowup introduces a new boundary hypersurface, 
which we refer to as the front face $\ff$. The other boundary faces are as follows. 
The union of ff and the right face $\rf$ is the lift of $\{x=0\}$. 
The union of ff and the left face $\lf$ is the lift of $\{\wx=0\}$. 
The union of ff and the temporal face $\tf$ is the lift of $\{\mu=\infty\}$.  \medskip

The actual blowup space $\mathscr{M}^2_b$ is obtained by a blowup of 
$[\R^+\times M^2; \mathscr{C}]$ along the lift of the diagonal $\mathscr{D}$. 
The resulting blowup space $\mathscr{M}^2_b$ is defined as before by cutting out 
the submanifold and replacing it with its inward-pointing spherical normal bundle. This second 
blowup introduces an additional boundary hypersurface td, the temporal diagonal. 
$\mathscr{M}^2_b$ is a manifold with boundaries and corners, illustrated below.

\begin{figure}[h]
\begin{center}
\begin{tikzpicture}
\draw (0,0.7) -- (0,2);
\draw (-0.7,-0.5) -- (-2,-1);
\draw (0.7,-0.5) -- (2,-1);
\draw (0,0.7) .. controls (-0.5,0.6) and (-0.7,0) .. (-0.7,-0.5);
\draw (0,0.7) .. controls (0.5,0.6) and (0.7,0) .. (0.7,-0.5);
\draw (-0.7,-0.5) .. controls (-0.5,-0.6) and (-0.4,-0.7) .. (-0.3,-0.7);
\draw (0.7,-0.5) .. controls (0.5,-0.6) and (0.4,-0.7) .. (0.3,-0.7);
\draw (-0.3,-0.7) .. controls (-0.3,-0.3) and (0.3,-0.3) .. (0.3,-0.7);
\draw (-0.3,-1.4) .. controls (-0.3,-1) and (0.3,-1) .. (0.3,-1.4);
\draw (0.3,-0.7) -- (0.3,-1.4);
\draw (-0.3,-0.7) -- (-0.3,-1.4);

\draw [very thick] (-0.2,0.85) -- (-0.2,1.35);
\draw [very thick] (-0.2,1.35) -- (-0.15,1.15);
\draw [very thick] (-0.2,1.35) -- (-0.25,1.15);
\draw [very thick] (-0.2,0.85) -- (-0.6,0.65);
\draw [very thick] (-0.6,0.65) -- (-0.5,0.75);
\draw [very thick] (-0.6,0.65) -- (-0.45,0.68);

\node at (-0.8,0.65) {$\xi$};
\node at (-0.45,1.35) {$\rho$};

\draw [very thick] (0.85,-0.4) -- (1.3,-0.58);
\draw [very thick]  (1.3,-0.58) -- (1.1,-0.55);
\draw [very thick]  (1.3,-0.58) -- (1.15,-0.45);
\draw [very thick]  (0.85,-0.4) -- (0.75,0.15);
\draw [very thick]  (0.75,0.15) -- (0.72,-0.05);
\draw [very thick]  (0.75,0.15) -- (0.85,-0.03);

\node at (0.95,0.15) {$\tau$};
\node at (1.5,-0.58) {$\wx$};

\node at (-1.3,0.1) {\large{lf}};
\node at (1.2,0.9) {\large{rf}};
\node at (1.1, -1.2) {\large{tf}};
\node at (-1.1, -1.2) {\large{tf}};
\node at (0, -1.7) {\large{td}};
\node at (0,0.1) {\large{ff}};
\end{tikzpicture}
\end{center}
\caption{The blowup $\mathscr{M}^2_b=[[\R^+\times M^2, \mathscr{C}], \mathscr{D}]$.}
\end{figure}

The blowup $\mathscr{M}^2_b$ is equipped with the canonical `blow-down map'
$\beta: \mathscr{M}^2_b\to \R^+\times M^2$, which we discuss below. \medskip

Denote by $Y:=\{(\mu, p, \widetilde{p}) \in \R^+\times M^2 \mid p=\widetilde{p}\}$ the diagonal 
submanifold and denote by $\beta^*Y$ its lift to $\mathscr{M}^2_b$. We denote by $\mathscr{A}_{\textup{phg}}^{l,p}(\mathscr{M}^2_b, \beta^*Y)$
the space of distributions that lift\footnote{A distribution $\w$ on $\R^+\times M^2$ lifts to a distribution $\beta^*\w$ on $\mathscr{M}^2_b$
by requiring $(\beta^*\w) (f \circ \beta) := \w (f)$ on any test function $f \in C^\infty_0(\R^+\times M^2)$. Note that a
(compactly supported) test function on $\mathscr{M}^2_b$ can always be written as $f\circ \beta$ for some test function $f$ on $\R^+\times M^2$.} 
to polyhomogeneous conormal distributions on the blowup space $\mathscr{M}^2_b$,
with leading order $(-m+l)$ at the front face $\ff$, leading order $(-m+p)$ at the temporal diagonal $\td$, 
index sets $(\N_0, \N_0)$ at the left and right boundary faces, vanishing to infinite order at tf, 
conormal at the interior submanifold $\beta^*Y$.
The space of such distributions without a conormal singularity is denoted by 
$\mathscr{A}_{\textup{phg}}^{l,p}(\mathscr{M}^2_b)$.
\medskip  

A choice of projective coordinates on $\mathscr{M}^2_b$ is given as follows. 
Near the top corner of ff away from tf projective coordinates are given by
\begin{align}\label{top-coord}
\rho=\frac{1}{\mu}, \  \xi=\frac{x}{\rho}, \ \widetilde{\xi}=\frac{\wx}{\rho}, \ w= \frac{y-\wy}{\rho}, \ \wy,
\end{align}
where in these coordinates $\rho, \xi, \widetilde{\xi}$ are the defining functions of 
the faces ff, rf and lf respectively. For the bottom corner of ff near rf  
projective coordinates are given by
\begin{align}\label{right-coord}
\tau=(\mu \wx)^{-1}, \ s=\frac{x}{\wx}, \ u = \frac{y-\wy}{\wx}, \ \wx, \ \wy,
\end{align}
where in these coordinates $\tau, s, \wx$ are defining functions of tf, rf and ff respectively. 
For the bottom corner of ff near lf projective coordinates are obtained by interchanging 
the roles of $x$ and $\wx$. Projective coordinates on $\mathscr{M}^2_b$ near the top of td away 
from tf are given by 
\begin{align}\label{d-coord}
\eta=\tau, \ S =\frac{s-1}{\eta}, \ U = \frac{u}{\eta}, \ \wx, \ \wy.
\end{align}

The temporal face tf is not in this coordinate chart but corresponds to the limit $|(S,U)|\to \infty$. 
The boundary faces ff and td are defined by 
$\wx, \eta$, respectively. The blowup $\mathscr{M}^2_b$ is related to the original 
space $\R^+\times M^2$ via the obvious `blow-down map'
\[
\beta: \mathscr{M}^2_b\to \R^+\times M^2,
\]
which is in local coordinates simply the coordinate change back to $(1/\mu, (x,y), (\wx,\wy))$. 
The blowup $\mathscr{M}^2_b$ is similar to the blowup space construction for incomplete
conical singularities by Mooers \cite{Moo:HKA} with the difference that here the 
blowup is not parabolic in $\mu^{-1}-$direction. It will be crucial in the argument below to note that
the restriction of $\beta$ to ff is a fibration over the diagonal of $\partial M \times \partial M$.
with the fibre equal to the quarter sphere $\mathbb{S}^m_{++}$. Similarly, the restriction of 
$\beta$ to td is a vibration over the diagonal of $M \times M$ with the fibre equal to the hemisphere
$\mathbb{S}^m_+$. 

\begin{theorem}\label{2-3}
Consider the elliptic boundary value problem $(A,B)$ with a differential operator $A$
of order $q$ on a compact manifold $M$ of dimension $m$ with boundary $\partial M$. Then the 
resolvent kernel $K_{A_B}$ and the Schwartz kernel $K_A$ corresponding to an interior symbolic parametrix kernel
are both elements of $\mathscr{A}_{\textup{phg}}^{q, q}(\mathscr{M}^2_b, \beta^*Y)$, 
where $Y:=\{(\mu, p, \widetilde{p}) \in  \R^+\times M^2 \mid p=\widetilde{p} \}$ denotes the diagonal 
submanifold. Moreover, their difference $K_{A_B} - K_A =:K \in \mathscr{A}_{\textup{phg}}^{q, \infty}(\mathscr{M}^2_b)$.
\end{theorem}

\begin{proof}
For $R>0$ sufficiently large and $\mu\in \Gamma_R$, we may write according to 
\cite[(28)]{See:TRO} 
\begin{equation}
\textup{Op}(d_{-q-j}) = (2\pi)^{-m} \int_{\R^{m-1}} e^{i\langle y-\wy, \zeta\rangle}
\widetilde{d}_{-q-j} (x, y, \zeta, \wx, \mu) \, d\zeta, 
\end{equation}
where $\widetilde{d}_{-q-j}$ is homogeneous of degree $(-q-1)$ in $(x^{-1}, \wx^{-1}, \zeta, \mu)$.
Moreover, \cite[(29)]{See:TRO} asserts the following estimate
\begin{equation}\label{d-tilde}
\begin{split}
\left|x^i \wx^k\partial^\A_x\partial^\beta_{\wx} \partial^\gamma_{\zeta} \partial^\delta_\mu 
\widetilde{d}_{-q-j} (x, y, \zeta, \wx, \mu)\right| &\leq C \exp \left(-c(x+\wx)(|\zeta|+ \mu)\right) \\
&\times (|\zeta| + \mu)^{1-q-j-k-i+ \A+\beta-|\gamma|-\delta},
\end{split}
\end{equation}
with constants $c,C>0$. This estimate is stable under differentiation in $y\in \R^{m-1}$.
We may now study the asymptotics of the lift $\beta^*\textup{Op}(d_{-q-j})$ in the various 
projective coordinates near the front face of $\mathscr{M}^2_b$. For instance, in coordinates
\eqref{right-coord} we find
\begin{equation}
\beta^*\textup{Op}(d_{-q-j}) = (2\pi)^{-m} \wx^{-m+q+j} 
\int_{\R^{m-1}} e^{i\langle u, \nu\rangle}
\widetilde{d}_{-q-j} (s, y, \nu, 1, \tau^{-1}) \, d\nu. 
\end{equation}
Hence the lift $\beta^*\textup{Op}(d_{-q-j})$ is of order $(-m+q+j)$ at the front face $\ff$ ($\wx\to 0$), 
and smooth at $\rf$ ($s\to 0$). In view of the estimate \eqref{d-tilde} the expression is also vanishing to 
infinite order at the temporal face $\tf$ ($\tau \to 0$). Similarly, in coordinates \eqref{d-coord} 
\begin{equation}
\beta^*\textup{Op}(d_{-q-j}) = (2\pi)^{-m} (\eta\wx)^{-m+q+j} 
\int_{\R^{m-1}} e^{i\langle U, \nu\rangle}
\widetilde{d}_{-q-j} (S+\eta^{-1}, y, \nu, \eta^{-1}, 1) \, d\nu. 
\end{equation}
Hence the lift $\beta^*\textup{Op}(d_{-q-j})$ is of order $(-m+q+j)$ at the front face $\ff$ ($\wx\to 0$). 
In view of the estimate \eqref{d-tilde} the expression is also vanishing to 
infinite order at the temporal face $\tf$ ($|(S,U)|\to \infty$) as well as temporal diagonal $\td$ ($\eta \to 0$). 
\medskip

Summarizing we have shown $\beta^*\textup{Op}(d_{-q-j}) \in \mathscr{A}_{\textup{phg}}^{q+j, \infty}(\mathscr{M}^2_b)$. 
Similar arguments applied to the classical symbol expansion of the interior parametrix $K_A=(A+\mu^q)^{-1}$
justify $K_A \in \mathscr{A}_{\textup{phg}}^{q, q}(\mathscr{M}^2_b, \beta^*Y)$. Since $\beta^*\mu = \rho_\ff \rho_\td \rho_\tf$,
up to multiplication with a bounded function, we infer from Theorem \ref{thm-seeley}
\begin{equation}
\beta^*K_{A_B} = \beta^*K_A + \sum_{j=0}^{N-1} \beta^*\textup{Op}(d_{-q-j})
+ (\rho_\ff \rho_\td \rho_\tf)^{-m+q+N} Q_N, 
\end{equation}
as $(\rho_\ff, \rho_\td, \rho_\tf) \to 0$, where $Q_N$ denotes some conormal distribution 
on $\mathscr{M}^2_b$ which is bounded uniformly in $N$ at all the boundary faces.
Taking the limit $N\to \infty$ proves the statement.
\end{proof}

%%%%%%%%%%%%%%%%%%%%%%%
\section{Composition of polyhomogeneous Schwartz kernels} \label{comp-sec}
%%%%%%%%%%%%%%%%%%%%%%%%%

Let $X$ and $X'$ be two compact manifolds with corners, and let $f: X \to X'$ 
be a smooth map.  Let $\{H_i\}_{i\in I}$ and $\{H'_j\}_{j\in J}$ be enumerations of the codimension one boundary faces of $X$ and $X'$,
respectively, and let $\rho_i$, $\rho'_j$ be global defining functions for $H_i$, resp.\ $H'_j$. We say that the map $f$
is a $b$-map if for all $j\in J$ there exists a smooth positive function $a_j$ such that
\[
f^* \rho_j' = a_{j} \prod_{i\in I} \rho_j^{e(i,j)}, \quad e(i,j) \in \mathbb{N} \cup \{0\}.
\]
The map $f$ is called a $b$-submersion if $f_*$ induces a surjective map between the $b$-tangent bundles
of $X$ and $X'$. The notion of $b$-tangent bundles has been introduced in \cite[Lemma 2.5]{Mel:TAP}. Assume moreover that
for each $j$ there is at most one $i$ such that $e(i,j) \neq 0$. In other words no
submanifold in $X$ gets mapped to a corner in $X'$. Under this condition the $b$-submersion $f$ is called a $b$-fibration. 
\medskip

Suppose that $\nu_0$ is a density on $X$ which is smooth up to all boundary faces and everywhere nonvanishing. 
A smooth $b$-density $\nu_b$ is, by definition, any density of the form $\nu_b =\nu_0 (\Pi \rho_i)^{-1}$.  
Let us fix smooth $b$-densities $\nu_b$ on $X$ and $\nu_b'$ on $X'$. 

\begin{prop}\label{push}\cite[The Pushforward Theorem]{Mel:COC}
Let $f_b:X\to X'$ be a $b$--fibration.
Let $u$ be a polyhomogeneous function on $X$ with index sets  $E_i$ the faces $H_i$ of $X$.  Suppose that each $(z,p) \in E_i$
has $\mbox{Re}\, z > 0$ if $e(i,j) = 0$ for all $j\in J$.
Then the pushforward $f_* (u \nu_b)$ is well-defined and equals $h \nu_b'$ where $h$ is polyhomogeneous on $X'$ and has
an index family $f_b(\mathcal{E})$ given by an explicit formula in terms of the index family $\mathcal{E}$ for $X$.
\end{prop}

Rather than giving the formula for the image index set in general, we provide the index image set in a specific setup,
enough for the present situation.  If $H_{i_1}$ and $H_{i_2}$ are both mapped to a face $H'_j$, and if 
$H_{i_1} \cap H_{i_2} = \emptyset$,
then they contribute the index set $E_{i_1} + E_{i_2}$ to $H'_j$. If they do intersect, however, then the contribution is the
extended union $E_{i_1} \overline{\cup} E_{i_2}$
\begin{align*}
E_{i_1} \overline{\cup} E_{i_2} := E_{i_1} \cup E_{i_2} \cup \{((z, p + q + 1): \exists \, (z,p) \in E_{i_1},\ 
\mbox{and}\  (z,q) \in E_{i_2} \}.
\end{align*}

We now employ the Pushforward theorem to establish the following fundamental composition result.
Let $Y:=\{(\mu, p, \widetilde{p}) \in \R^+\times M^2 \mid p=\widetilde{p}\}$ denote the diagonal 
submanifold. Consider any $K_a \in \mathscr{A}_{\textup{phg}}^{\ell,k}(\mathscr{M}^2_b, \beta^*Y)$ 
and $K_b \in \mathscr{A}_{\textup{phg}}^{\ell',\infty}(\mathscr{M}^2_b)$. 
Their composition is defined by 
\begin{equation}
K_{c}(p, \widetilde{p};\mu) = \int_M K_a(p, p';\mu) K_b(p',\widetilde{p};\mu) \, \textup{dvol}_M(p').
\label{comp1}
\end{equation}

\begin{prop}
\begin{align}
\mathscr{A}_{\textup{phg}}^{\ell,k}(\mathscr{M}^2_b, \beta^*Y) \circ \mathscr{A}_{\textup{phg}}^{\ell',\infty}(\mathscr{M}^2_b) 
\subset \mathscr{A}_{\textup{phg}}^{\ell+\ell', \infty}(\mathscr{M}^2_b).
\label{composition}
\end{align}
\end{prop}

\begin{proof}
Consider $K_a \in \mathscr{A}_{\textup{phg}}^{l,k}(\mathscr{M}^2_b, \beta^*Y)$ 
and $K_b \in \mathscr{A}_{\textup{phg}}^{l',\infty}(\mathscr{M}^2_b)$. Their 
composition $K_c = K_a \circ K_b$ is defined in \eqref{comp1}.
This expression can be rephrased in geometric terms. 
Consider the space $\R^+_{1/\mu}\times M^3_{(p,p',\widetilde{p})}$, and the three projections
\begin{equation}
\begin{split}
&\pi_c :\R^+_{1/\mu}\times M^3_{(p,p',\widetilde{p})} \to \R^+_{1/\mu}\times M^2_{(p,\widetilde{p})}, \\
&\pi_a: \R^+_{1/\mu}\times M^3_{(p,p',\widetilde{p})} \to \R^+_{1/\mu}\times M^2_{(p,p')}, \\
&\pi_b: \R^+_{1/\mu}\times M^3_{(p,p',\widetilde{p})} \to \R^+_{1/\mu}\times M^2_{(p',\widetilde{p})}.
\end{split}
\label{projections}
\end{equation} 
We reinterpret $K_a, K_b$ and $K_c$ as `right densities'
\begin{align*}
&K_a \equiv K_a(p, p';\mu) \, \textup{dvol}_M(p'), \\
&K_b \equiv K_b(p',\widetilde{p};\mu) \, \textup{dvol}_M(\widetilde{p}), \\
&K_c \equiv K_c(p, \widetilde{p};\mu) \textup{dvol}_M(\widetilde{p}).
\end{align*}
Then we can rewrite \eqref{comp1} as
\[
K_c = (\pi_c)_* \left( \pi_a^* K_a  \cdot \pi_b^* K_c \right).
\]
The basic idea in the proof of polyhomogeneity of $K_c$ is a construction of a triple-space
$\mathscr{M}^3_b$ which is a blowup of $\R^+_{1/\mu}\times M^3$ obtained 
by a sequence of blowups, designed such that there are maps
\[
\Pi_a, \Pi_c, \Pi_b:  \mathscr{M}^3_b \longrightarrow 
[\R^+_{1/\mu}\times M^2; \mathscr{C}]=:  \mathscr{M}^2_{rb}
\]
which `cover' the three projections defined above. 
The construction is reminiscent of the triple space construction for the heat space calculus for conical 
singularities, see \cite{Moo:HKA}, but differs from the latter since there is no convolution 
in the parameter $\mu^{-1}$ variable and the blowups are not parabolic in the $\mu^{-1}$
direction. On each copy of $M$ we use the local coordinates 
$p=(x,y), p'=(x',y'), \widetilde{p}=(\wx, \wy) \in \R^+\times \R^{m-1}$ near the 
boundary $\partial M$ with $(x,x',\wx)$ being the three copies of the boundary defining function. 
First we blow up the submanifold
\begin{align*}
F&=\{(1/\mu, x,y, x',y', \wx, \wy) \mid \mu = \infty, x=x'=\wx=0, y=y'=\wy\} \\ &= 
\pi_a^{-1}\mathscr{C} \cap \pi_b^{-1}\mathscr{C} \cap \pi_c^{-1}\mathscr{C},
\end{align*}
which is the intersection of all highest codimension corners $\mathscr{C}$ introduced in \eqref{D-C}, 
pulled back under the three 
projections to $\R^+\times M^3$. Then we blow up the resulting space $[\R^+\times M^3; F]$ at the 
interior lifts of each of the three submanifolds
\begin{equation}
\begin{split}
F_c&=\pi_c^{-1}\mathscr{C} = \{(1/\mu, x,y, x',y', \wx, \wy) \mid \mu = \infty, x=\wx=0, y=\wy\}, \\
F_a&=\pi_a^{-1}\mathscr{C} = \{(1/\mu, x,y, x',y', \wx, \wy) \mid \mu = \infty, x=x'=0, y=y'\}, \\
F_b&=\pi_b^{-1}\mathscr{C} = \{(1/\mu, x,y, x',y', \wx, \wy) \mid \mu = \infty, x'=\wx=0, y'=\wy\}.
\end{split}
\end{equation} 
Identifying notationally each $F_{a,b,c}$ with their lifts to $[\R^+\times M^3, F]$, we may 
altogether define the triple space
\[
\mathscr{M}^3_b := \left[\left[\R^+\times M^3; F\right]; F_a; F_b; F_c\right].
\]

If we ignore the $\mu^{-1}-$direction, the spacial part of $\mathscr{M}^3_b$ 
can be visualized as below.
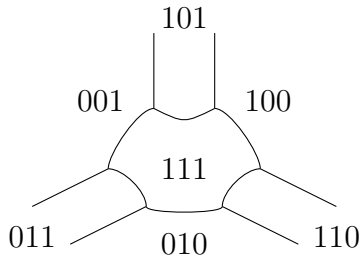
\begin{figure}[h]
\begin{center}
\begin{tikzpicture}

\draw  (-1,0) .. controls (-0.8,0) and (-0.5,-0.3) .. (-0.5,-0.5);
\draw  (1,0) .. controls (0.8,0) and (0.5,-0.3) .. (0.5,-0.5);
\draw  (-0.4,0.8) .. controls (0,0.6) and (0,0.6) .. (0.4,0.8);

\draw  (-1,0) -- (-2,-0.5);
\draw  (-0.5,-0.5) -- (-1.5,-1);

\draw  (1,0) -- (2,-0.5);
\draw  (0.5,-0.5) -- (1.5,-1);

\draw   (-0.4,0.8)  --  (-0.4,1.8) ;
\draw   (0.4,0.8)  --  (0.4,1.8);

\draw  (-1,0) .. controls (-1,0.2) and (-0.6,0.8) ..  (-0.4,0.8);
\draw  (1,0) .. controls (1,0.2) and (0.6,0.8) ..  (0.4,0.8);
\draw  (-0.5,-0.5) .. controls (-0.5,-0.6) and (0.5,-0.6) .. (0.5,-0.5);

\node at (0,0) {111};
\node at (-2,-0.9) {011};
\node at (2,-0.9) {110};
\node at (0,2) {101};
\node at (0,-1) {010};
\node at (-1.1,0.9) {001};
\node at (1.1,0.9) {100};
\end{tikzpicture}
\end{center}
\label{triple-space}
\caption{The spacial component of the triple space $\mathscr{M}^3_b$.}
\end{figure} 

Here, $(101), (011)$ and $(110)$ label the boundary faces created by blowing up $F_c, F_b$ and $F_a$, respectively. 
The face $(111)$ is the front face introduced by blowing up $F$. We denote the defining function for the face $(ijk)$  
by $\rho_{ijk}$. The triple space comes with a natural blowdown map 
$\beta^{(3)}: \mathscr{M}^3_b \to \R^+ \times M^3$, which as in the discussion of $\mathscr{M}^2_b$
amounts in local coordinates to a coordinate change back to $(x,y, x', y',\wx, \wy, \mu)$.
\medskip

Now consider the projections $\pi_c$, $\pi_a$ and $\pi_b$ introduced in \eqref{projections}.  
These induce projections $\Pi_c$, $\Pi_a$ and $\Pi_b$ from 
$\mathscr{M}^3_b$ to the reduced blowup space $\mathscr{M}^2_{rb}$. 
It is not hard to check that the choice of submanifolds $F, F_{a,b,c}$ that have been blown up ensures that these 
projections are in fact $b$-fibrations. \medskip

Denote defining functions for the right, front and left faces of each copy of $\mathscr{M}^2_{rb}$ by 
$\{\rho_{10}, \rho_{11}, \rho_{01}\}$, respectively.  These lift via the projections according to the following rules
(modulo multiplication by a non-vanishing function)
\begin{equation}
\begin{split}
&\Pi_c^*(\rho_{ij})=\rho_{i0j}\rho_{i1j}, \\ 
&\Pi_a^*(\rho_{ij})=\rho_{ij0}\rho_{ij1}, \\ 
&\Pi_b^*(\rho_{ij})=\rho_{0ij}\rho_{1ij}.
\end{split}
\label{RLC}
\end{equation}

Now consider the behaviour in the parameter $\mu^{-1}$-direction. Let $\tau$ be a defining function 
for the boundary face in $\mathscr{M}^3_b$ which is mapped onto $\{\mu=\infty\}$ by the blowdown map. 
Since each $F, F_{a,b,c}$ is a submanifold of $\{\mu=\infty\}$, we find 
\begin{align}
\label{beta-tau}
\beta_{(3)}^*\mu^{-1} = \tau \rho_{111} \rho_{110} \rho_{101} \rho_{011}.
\end{align} 

Let $\beta_{(2)}: \mathscr{M}^2_{rb} \to \R^+ \times M^2$ be the 
blowdown map for the reduced blowup space. Then $\beta_{(2)}^*\mu^{-1}=T \rho_{11}$, 
where $T$ is a defining function for the temporal face tf in $\mathscr{M}^2_{rb}$. Note 
that $\beta_{(2)} \circ \Pi_{a,b,c}= \pi_{a,b,c}\circ \beta_{(3)}$ and hence 
acting on functions on $\R^+ \times M^2$ we have 
\[
\Pi^*_{a,b,c} \circ \beta_{(2)}^* = \beta_{(3)}^* \circ \pi^*_{a,b,c}.
\]
Consequently, in view of \Eqref{RLC} and \Eqref{beta-tau}, we conclude
(modulo multiplication by a non-vanishing function)
\begin{equation}
\begin{split}
&\Pi_c^*(T)=\tau \rho_{110} \rho_{011}, \\
&\Pi_a^*(T)=\tau \rho_{101} \rho_{011}, \\
&\Pi_b^*(T)=\tau \rho_{101} \rho_{110}.
\end{split}
\label{TT}
\end{equation}

Using these data, we may now prove the anticipated composition formula. 
Consider the `right densities',
$K_a(x,y,x',y';\mu) dx'dy'$ and $K_b(x',y',\wx,\wy; \mu) d\wx d\wy$. 
Their product is given by 
\[
K_a(x,y,x',y';\mu) \cdot K_b(x',y',\wx,\wy; \mu)  \, dx'\, dy'\, d\wx \, d\wy
\]
Its integral over $dx'dy'$ gives $K_{c}(x,y,\wx, \wy; \mu) d\wx \, d\wy$. 
To put this into the same form required in the pushforward theorem, 
write $t= \mu^{-1}$ and multiply this expression by $dt \, dx\, dy$.
\medskip 

Blowing up a submanifold of codimension $n$ amounts in local coordinates to introducing polar coordinates, so that 
the coordinate transformation of a density leads to $(n-1)^{\mathrm{st}}$ power of the radial function, which 
is a defining function of the corresponding front face. Hence we compute the lift
\begin{equation}
\begin{split}
&\beta_{(3)}^* (dt \, dx\, dy\, dx'\, dy'\, d\wx \, d\wy) \\
&=\rho_{111}^{3+2(m-1)}\rho_{101}^{2+(m-1)}\rho_{110}^{2+(m-1)}\rho_{011}^{2+(m-1)} 
\nu^{(3)}\\ &=\rho_{111}^{3+2(m-1)}\rho_{101}^{2+(m-1)}\rho_{110}^{2+(m-1)}\rho_{011}^{2+(m-1)} 
\tau \left( \Pi \rho_{ijk} \right) \nu^{(3)}_b,
\end{split}
\label{lift1}
\end{equation}
where $\nu^{(3)}$ is a density on $\mathscr{M}^3_b$, smooth up to all boundary faces and everywhere 
nonvanishing; $\nu^{(3)}_b$ is a $b$-density, obtained from $\nu^{(3)}$ by dividing 
by a product of all defining functions on $\mathscr{M}^3_b$; 
and $\left( \Pi \rho_{ijk} \right)$ is a product over all $(ijk)\in \{0,1\}^3$. 
Set $\kappa_a =  \beta_{(2)}^* K_a$ and $\kappa_b =  \beta_{(2)}^* K_b$. 
Since $\kappa_b$ is vanishing to infinite order as $T \to 0$, its lift
$\Pi_b^*\kappa_b$ vanishes to infinite order in $\tau \rho_{110} \rho_{101}$ by \Eqref{TT}.  
\medskip

Since $\kappa_a$ is not polyhomogeneous on the reduced blowup space $\mathscr{M}^2_{rb}$,
the lift $\Pi_a^*\kappa_a$ is not polyhomogeneous in $\tau$. However, due to infinite 
order vanishing of $\Pi_b^*\kappa_b$ in $\tau$, their product 
$\Pi_a^*\kappa_a \cdot \Pi_b^*\kappa_b$ is polyhomogeneous and vanishing 
to infinite order in $\tau \rho_{110} \rho_{101} \rho_{011}$. We obtain
\[
\Pi_a^*\kappa_a \cdot \Pi_b^*\kappa_b \, \beta_{(3)}^* (dt \, dx\, dy\, dx'\, dy'\, d\wx \, d\wy) 
 =\rho_{111}^{\ell + \ell' + 1} \left( \Pi \rho_{ijk} \right) G \nu_b^{(3)},
\]
where $G$ is a bounded polyhomogeneous function on $\mathscr{M}^3_b$, 
vanishing to infinite order in $(\tau \rho_{110} \rho_{101} \rho_{011})$, 
with index sets $\N_0$ at the faces $(001)$, $(100)$ and $(010)$.
Applying the Pushforward Theorem now gives
\begin{equation}
\begin{split}
&\left(\Pi_c\right)_*\left(\Pi_a^*\kappa_a \cdot \Pi_b^*\kappa_b 
\, \beta_{(3)}^* (dt \, dx\, dy\, dx'\, dy'\, d\wx \, d\wy)  \right) \\
& = \beta_{(2)}^*\left(K_{c} \, dt\, dx\, dy\, d\wx \, d\wy\right) 
= \rho_{11}^{2+\ell+\ell'} \, G' \, \nu^{(2)}_b,
\end{split}
\label{lift3}
\end{equation}
where $\nu^{(2)}_b$ is a $b$-density on $\mathscr{M}^2_{rb}$ 
and $G'$ is a bounded polyhomogeneous function on $\mathscr{M}^2_{rb}$, 
which vanishes to infinite order in $T$, and has the index set $\N_0$ 
at the left and right boundary faces. By \cite[Proposition B7.20]{EMM:ROT} the 
pushforward is smooth across $\beta^*Y$.
\medskip

Note also that the pushforward by $\Pi_c$ does not introduce logarithmic 
terms in the front face expansion of $\kappa_{c}$, since 
the kernel on $\mathscr{M}^2_{rb}$ is vanishing to infinite order at $(101)$. 
Hence, for $\kappa_A$ and $\kappa_B$ with integer exponents in their front face expansions, 
the same holds for their composition. \medskip

By an argument similar to \Eqref{lift1}, we compute
\begin{equation}
\begin{split}
\beta_{(2)}^*(dt\, dx\, dy\, d\wx \, d\wy)=
\rho_{11}^{m+1} \left(\rho_{10}\rho_{11}\rho_{01} T\right) \nu^{(2)}_b .
\end{split}
\label{lift4}
\end{equation}
Consequently, combining \Eqref{lift3} and \Eqref{lift4}, we deduce that 
$\beta_{(2)}^* K_{c}$ vanishes to infinite order in $T$, is of leading order $(-m+ \ell + \ell')$ at 
the front face and has the index sets $\N_0$ 
at the left and right boundary faces. 
This proves the statement.
\end{proof}

%%%%%%%%%%%%%%%%%%%%%%%
\section{Multi-parameter resolvent trace expansion} \label{multi-sec}
%%%%%%%%%%%%%%%%%%%%%%%%%

\subsection{Initial parametrix and Neumann series}
In this section we construct the Schwartz kernel of $(A_B(\lambda)+z^q)^{-1}$
as a polyhomogeneous distribution in an open neighborhood of ff in $\mathscr{M}^2_b$.
The microlocal description of the resolvent kernels in Theorem \ref{2-3}
does not apply directly to $A_B(\lambda)$, since the smooth potentials $V_1,..,V_n$
need not be constant along $\partial M$. Hence we consider for each $y_0\in \partial M$
the operator $A_B(\lambda)_{y_0} := A_B + \sum_{k=1}^n \lambda^q_k V_k(y_0),$
with the corresponding resolvent for $z\in \Gamma$ sufficiently large
\begin{align}
K_{y_0} := (A_B(\lambda)_{y_0} + z^q)^{-1}(x,\wx, y, \wy).
\end{align}
The front face ff and its open $\varepsilon$-neighborhood in $\mathscr{M}^2_b$
are fibrations over $\partial M$ with fibres $\mathbb{S}^m_{++}$ and $[0,\varepsilon) 
\times \mathbb{S}^m_{++}$, respectively. Theorem \ref{2-3} asserts that 
$K_{y_0}(\cdot, y_0)$ lifts to a polyhomogeneous conormal distribution on 
the $y_0$-fibre with $\mu = |(\lambda, z)|^{-1}$. Consequently we may define
a distribution $K_0 \in \mathscr{A}_{\textup{phg}}^{q,q}(\mathscr{M}^2_b, \beta^*Y)$
by (e.g. in projective coordinates \eqref{right-coord})
\begin{align}
K_0(\tau, s,\wx, u, y_0) := K_{y_0}(s\wx, \wx, y_0 + \wx u, y_0).
\end{align}
By construction, $BK_0=0$ and $K_0$ solves the resolvent equation\footnote{$I$ denotes the identity operator.} $(A_B(\lambda)+z^q) K_0=I+P$
up to an error term $P$ that vanishes at ff to one order higher, more precisely is of order $(q+1)$ at ff.
As in Theorem \ref{2-3}, we may separate $K_0$ into an interior
and boundary contribution
\begin{align}
K_0 = K_i + K_b, \quad K_i \in \mathscr{A}_{\textup{phg}}^{q,q}(\mathscr{M}^2_b, \beta^*Y), 
K_b \in \mathscr{A}_{\textup{phg}}^{q,\infty}(\mathscr{M}^2_b).
\end{align}
Formally, $K_0$ is corrected to an exact solution of $(A_B(\lambda)+z^q) K=I$
by convolution with a Neumann series of the remainder $R$
\begin{equation}
\begin{split}
K &= K_0 (I+P)^{-1} = K_0 \sum_{j=0}^\infty (-1)^j P^j
\\ &= \sum_{j=0}^\infty (-1)^j \left( K_0 P^j - K_i P^j \right) + \sum_{j=0}^\infty (-1)^j 
K_i P^j \\ &=: \sum_{j=0}^\infty (-1)^j R^j_0 + \sum_{j=0}^\infty (-1)^j R^j_1 =: R_0 + R_1. 
\end{split}
\end{equation}
The Neumann series converge in the operator norm for $z\in\Gamma$ sufficiently large.
Note that the second term $R_1$ is an interior parametrix $(A(\lambda)+z^q)^{-1}$. 
Its multi-parameter expansion follows from the classical calculus of pseudo-differential operators 
with parameter, for a survey type exposition see for example \cite[Sec. 4 and 5]{Les:PDO}.
\medskip

More precisely, the differential expression $(A(\lambda)+z^q)$ 
is of order $q$, elliptic in the parametric sense with parameter 
$(\lambda, z)\in \Gamma^{n+1}$. We write $(A(\lambda)+z^q) \in \CL^{q}(M;\Gamma^{n+1})$.
By \cite[Sec. II.9]{Shu:POS}, its parametrix $R_1\in \CL^{-q}(M;\Gamma^{n+1})$, and the
$N$-th power $R_1^N \in \CL^{-qN}(M;\Gamma^{n+1})$. Consider $N\in \N$, such that 
$qN >m$. Fix any multiindex $\A \in \N_0^n$ and $\beta \in \N_0$. 
Then the Schwartz kernel of $R^N_1$, which we do not distinguish notationally from the operator, 
is continuous with an asymptotic expansion on the diagonal as $|(\lambda, z)| \to \infty$, 
$(\lambda, z) \in \Gamma^{n+1}$
\begin{align}\label{interior}
\partial^\A_\lambda \partial_z^\beta 
R^N_1(p,p;\lambda, z) \sim \sum_{j=0}^\infty f_j \left( p, \frac{(\lambda, z)}{|(\lambda, z)|} \right)
|(\lambda, z)|^{-qN - j -|\A| -\beta +m},
\end{align}
see \cite[Theorem 5.1]{Les:PDO}. The functions $f_j$
are smooth on $M \times (\Gamma^{n+1} \cap \mathbb{S}^n)$ and the expansion 
\Eqref{interior} is uniform over the compact manifold $M$. 

\subsection{Trace norm estimates and expansions} It remains to study $R_0$. 
Trace norm estimates of $R^j_0$ depend on the following basic result.

\begin{prop}\label{A-HS}
Consider $R \in \mathscr{A}_{\textup{phg}}^{\ell,\infty}(\mathscr{M}^2_b)$.
Then $R$ defines a Hilbert Schmidt operator with the Hilbert Schmidt norm
$$\|R(\cdot , \mu)\|_{HS} = O(\mu^{-\ell+m/2}), \ \mu \to \infty.$$
\end{prop}

\begin{proof}
We make the argument explicit in projective coordinates \eqref{right-coord}. 
$R$ lifts to a polyhomogeneous function on $\mathscr{M}^2_{rb}$ and 
\begin{align*}
&\beta^* (R^2 dx\, dy\, d\wx \, d\wy) = \wx^{- m + 2\ell} G \, ds \, d\wx \, du \, dy \\
& = \mu^{- 2 \ell + m} \tau^{m-2\ell} G \, ds \, d\wx \, du \, dy,
\end{align*}
where $G$ is bounded in $(\wx, s, u, y)$ and vanishing to infinite order as $\tau \to \infty$.
Similar estimates hold at the other two corners of the front face in the reduced blowup space $\mathscr{M}^2_{rb}$.
Consequently we obtain as $\mu \to \infty$
\begin{align*}
\|R(\cdot , \mu)\|^2_{HS} = \int_M \int_M R(p, \widetilde{p};\mu)^2 \textup{dvol}_M(p) \textup{dvol}_M(\widetilde{p}) 
= O(\mu^{-2\ell + m}).
\end{align*}
\end{proof}

\begin{cor}\label{R-expansion}
Let $M\in \N$ and $\gb\in\N_0$ be fixed. 
Write for any multiindex $\A \in \N_0^n$ 
$\partial^\A_\lambda = \partial^{\A_1}_{\lambda_1} \cdots \partial^{\A_n}_{\lambda_n}$.
Then for $\mu_0$ sufficiently large there exist constants $C>0$ and $0<q<1$ such that
for $N\ge M$ and $\mu\ge \mu_0$ the trace norms satisfy the following estimates
\begin{equation}
\begin{split}
&\bigl\| \, \pl^\A_\lambda \pl_z^\beta  R^{N}_0  \, \bigr\|_\tr 
\leq C \cdot N\cdot q^{N-M} \cdot \mu^{-M-|\ga|-\gb - q + \frac{m}{2}},
\\  &\Bigl\| \, \pl^\ga_\gl \pl_z^\gb \sum_{j=M}^\infty R^{j}_0 \, \Bigr\|_\tr =
 O\bl \mu^{-M-|\ga|-\gb - q + \frac{m}{2}}\br, \text{ as } \mu\to\infty.
\end{split}
\end{equation}
\end{cor}

\begin{proof} It suffices to consider the case $\ga=0^n $ and $\gb=0$. 
For elements $a$ and $b$ in a not 
necessarily commutative ring we have by induction the following identity
\begin{align}
(a+b)^N - b^N = \sum_{k=0}^{N-1} b^k a (a+b)^{N-k-1}.
\end{align}
By construction, $P = \lambda(V) K_0$, where in local coordinates
$(x,y)\in \mathscr{U}$ in an open collar neighborhood of the boundary $\lambda(V)
= \sum_{k=1}^n \lambda_k^q (V_k(x,y) - V_k(0,y))$. Consequently, we may write 
\begin{equation}
\begin{split}
R^N_0 &= K_{0} \circ \left( \lambda(V) K_{0} \right)^N 
- K_{i} \circ \left( \lambda(V) K_{i} \right)^N \\
&= \sum_{k=0}^{N-1} K_i \circ (\lambda (V) K_i )^k \circ K_b \, \lambda (V) \circ (\lambda(V) K_{0})^{N-k-1}
\\ &+ K_b \circ \left( \lambda(V) K_{0} \right)^N.
\end{split}
\end{equation}
Each of the $N$ summands of $R^{N}_0$ is of the form 
\begin{equation}\label{P-N}
P_N= Q_0 \circ \left(\prod\limits_{j=1}^N   \gl(V) Q_j\right),
\end{equation}
where each $Q_j, j=0,\ldots,N$, is either $K_{0}, K_i$ or $K_b$.
By construction, at least one of the kernels equals $K_b$. \medskip

By the classical resolvent decay, we have for $z\in \Gamma$ sufficiently large and 
for each $j$ an estimate of the operator norm 
\begin{align}
     \| \gl(V) Q_j \| \leq c <1.
\end{align}
Thus for any fixed $M\leq N$ we may estimate the trace norm of each summand $P_N$ by
\begin{equation}\label{EqTraNorEst}
  \bigl\| P_N \bigr\|_\tr \le 
  c^{N-M}\cdot \bigl\|  Q_0 \left(\prod_{j=1}^M \gl(V) Q_j \right) \bigr\|_\HS,
\end{equation}
where $\|\cdot\|_\tr, \|\cdot\|_\HS$ denote the trace norm resp.
the Hilbert-Schmidt norm. Note that for $\delta V_k(x,y):= V_k(x,y) - V_k(0,y)$ the lift 
$\beta^*\delta V_k$ vanishes to first order at ff and hence writing for any $\A \in \N_0^n$
$(\lambda^q)^\A:= \lambda_1^{q\A_1} \cdots \lambda_n^{q\A_n}$ and 
$|\A|:=\A_1+ \cdots + \A_n$ we find 
\begin{equation}\label{4-8}
Q_0 \left(\prod_{j=1}^M \gl(V) Q_j \right)  \in \bigoplus_{|\A|=M} 
(\lambda^q)^\A \mathscr{A}_{\textup{phg}}^{M(q+1)+q,\infty}(\mathscr{M}^2_b).
\end{equation}
By Proposition \ref{A-HS} we find
\begin{equation}
\bigl\| Q_0 \left(\prod_{j=1}^M \gl(V) Q_j \right) \bigr\|_\HS
\leq C_M \, \mu^{-M-q+m/2}.
\end{equation}
This proves the first and subsequently the second statements.
\end{proof}

We may now separate 
$$
R_0 = \sum_{j=0}^\infty (-1)^j R^j_0 = \sum_{j=0}^{M-1} (-1)^j R^j_0 + \sum_{j=M}^\infty (-1)^j R^j_0,
$$
and in view of Corollary \ref{R-expansion}, an asymptotic expansion of $\textup{Tr} R_0$
follows by studying the first sum and then taking $M\to \infty$ to improve the remainder estimate.
Asymptotics of the first sum follows from the next proposition.

\begin{prop}
\label{trace-boundary}
For any $R \in 
\mathscr{A}_{\textup{phg}}^{\ell,\infty}(\mathscr{M}^2_b)$
we find 
\begin{align}
\Tr R(\cdot, \mu)  \sim \sum_{j=0}^\infty 
a_j \mu^{-1-j-\ell + m}, \ \mu \to \infty.
\end{align} 
\end{prop}

\begin{proof}
Consider $Y:=\{(\mu, p, \widetilde{p}) \in \R^+\times M^2 \mid p=\widetilde{p}\}$.
The lift $\beta^* R$ does not have a conormal singularity and restricts to a polyhomogeneous distribution on
$\beta^*Y \subset \mathscr{M}^2_b$, which itself is a blowup of 
$\R^+_{1/\mu} \times M$ at the highest codimension corner, with the blowdown map denoted 
by $\beta_{Y}$. We refer to the restrictions of ff, lf and td in 
$\mathscr{M}^2_b$ to $\beta^*Y$ as the front face, left face and 
temporal diagonal again. \medskip

The restriction of $\beta^*R$ to $\beta^*Y$ is polyhomogeneous 
of leading order $(-m+\ell)$ at the front face, index set $\N_0$ 
at the left face and vanishes to infinite order at the temporal diagonal.
Consider the obvious projection $\pi: \R^+ \times M \to \R^+$. Then
($t=1/\mu$)
\begin{align*}
\Tr R(\cdot, \mu) \, dt = (\pi \circ \beta_Y)_* \left(\beta_Y^* R|_Y
dt \, \textup{dvol}_M\right).
\end{align*}
Note that 
\[
\beta_Y^* \left(R|_Y
dt \, \textup{dvol}_M\right) = \rho_\ff^{-m+\ell+2} G \nu_b.
\]
where $\nu_b$ is a $b$-density on $\beta^*Y$, $\rho_\ff$ is the 
defining function of the front face, $G$ is a bounded polyhomogeneous 
distribution with index set $\N_0$ at the left face and vanishes to infinite order at the temporal diagonal 
in $\beta^*Y$. By the Theorem \ref{push} we find as $t\to 0$
\begin{align*}
 (\pi \circ \beta_Y)_* \left( \beta_Y^* R|_Y
dt \, \textup{dvol}_M\right) \sim \sum_{j=0}^\infty t^{-m+\ell +2 +j} \left(t^{-1}dt\right).
\end{align*}
This proves the statement.
\end{proof}

\begin{cor} \label{R-0}
Consider any multiindex $\A \in \N_0^n$ and $\beta \in \N_0$.
For each $i \in \N_0$ there exist  $h_i \in C^\infty (\Gamma^{n+1} \backslash \{0\})$, homogeneous 
in $(\lambda, z)$ of order $(-1 -q+m -i)$, such that in the notation of Corollary \ref{R-expansion}
$$
\partial^\A_\lambda \partial_z^\beta  \Tr R_0 \sim 
\sum_{i=0}^\infty h_{i+|\A|+\beta}(\lambda, z),
\ \textup{as} \ \mu \to \infty.
$$
\end{cor}

\begin{proof}
As before in Corollary \ref{R-expansion} it suffices to consider $\A=0^n, \beta =0$.
As observed in \eqref{4-8}, we find by construction
$$
R_0^j = \sum_{|\A|=j}^n Q_{j,\A} 
\in \bigoplus_{|\A|=M} (\lambda^q)^\A 
\mathscr{A}_{\textup{phg}}^{j(q+1)+q,\infty}(\mathscr{M}^2_b).
$$
For each $Q_{j,\A} \in  (\lambda^q)^\A 
\mathscr{A}_{\textup{phg}}^{j(q+1)+q,\infty}(\mathscr{M}^2_b)$ we find by Proposition \ref{trace-boundary} 
as $\mu \to \infty$
\begin{align*}
\Tr Q_{j,\A}(\mu) \sim \sum_{i=0}^\infty b_i \left(\frac{(\lambda^q)^\A}{\mu^{q|\A|}}\right) \mu^{-1 - i - j - q+m}.
\end{align*}
Altogether we obtain the following polyhomogeneous multi-parameter trace expansion for any $M\in \N$
$$
\Tr \sum_{j=0}^{M-1} (-1)^j R^j_0 \sim \sum_{i=0}^\infty h'_{i}(\lambda, z),
\ \textup{as} \ \mu \to \infty,
$$
where each $h'_{i}(\lambda, z)$ is homogeneous in $(\lambda, z)$ of order $(-1-q+m-i)$.
Taking $M\to \infty$ we derive by Corollary \ref{R-expansion} the stated full polyhomogeneous expansion.
\end{proof}

Together with the asymptotics of the interior contribution \eqref{interior} this leads to our final main result.

\begin{theorem}
Consider any multiindex $\A \in \N_0^n$ and $\beta \in \N_0$. Fix $N\in \N$ such that 
$qN >m$. Then there exist  $e_i \in C^\infty (M \times (\Gamma^{n+1} \cap \mathbb{S}^n))$,
such that
$$\partial^\A_\lambda \partial_z^\beta \Tr (A_B(\lambda) + z^q)^{-N}
\sim \sum_{j=0}^\infty e_j \left(\frac{(\lambda, z)}{|(\lambda, z)|} \right)
|(\lambda, z)|^{-1-qN - j -|\A| -\beta +m},$$
where $e_0$ is purely a boundary contribution.
\end{theorem}

\begin{proof}
The statement follows from Corollary \ref{R-0}, \Eqref{interior} and
\begin{equation}
\begin{split}
\partial^\A_\lambda \partial_z^\beta (A_B(\lambda) + z^q)^{-N} = 
\partial^\A_\lambda \partial_z^\beta \left(-\frac{1}{q} z^{1-q} \partial_z\right)^{N-1}
(A_B(\lambda) + z^q)^{-1} \\ = \partial^\A_\lambda \partial_z^\beta \left(-\frac{1}{q} z^{1-q} \partial_z\right)^{N-1}
(R_0 + R_1).
\end{split}
\end{equation}
\end{proof}

\section*{Acknowledgements}
The interest in multi-parameter resolvent trace expansions 
arose in a joint project with Matthias Lesch, whom the author thanks for various central discussions.
The author also gratefully acknowledges helpful discussions with Benedikt Sauer, 
Leonid Friedlander and Rafe Mazzeo. The author is also indebted to the anonymous
referee for the very detailed reading of the manuscript. 
The author was supported by the Hausdorff Center for Mathematics.

\def\cprime{$'$}
\providecommand{\bysame}{\leavevmode\hbox to3em{\hrulefill}\thinspace}
\providecommand{\MR}{\relax\ifhmode\unskip\space\fi MR }
\providecommand{\MRhref}[2]{%
  \href{http://www.ams.org/mathscinet-getitem?mr=#1}{#2}
}
\providecommand{\href}[2]{#2}


\begin{thebibliography}{10}

\bibitem[\textsc{EMM91}]{EMM:ROT}
\textsc{C.~L. Epstein}, \textsc{R.~B. Melrose}, and \textsc{G.~A. Mendoza},
  \emph{Resolvent of the {L}aplacian on strictly pseudoconvex domains}, Acta
  Math. \textbf{167} (1991), no.~1-2, 1--106. \MR{1111745 (92i:32016)}


\bibitem[\textsc{LeVe13}]{LesVer:RSD}
\textsc{M. Lesch and B. Vertman} \emph{Regularized sum of zeta-determinants},
arXiv:1306.0780 [math.SP] (2013), accepted, to appear in Math. Annalen (2014)

\bibitem[\textsc{Les10}]{Les:PDO}
\textsc{M. Lesch}, \emph{Pseudodifferential operators and regularized traces}, Motives,
  quantum field theory, and pseudodifferential operators, Clay Math. Proc.,
  vol.~12, Amer. Math. Soc., Providence, RI, 2010, pp.~37--72.
\texttt{arXiv:0901.1689 [math.OA]}


\bibitem[\textsc{Mel92}]{Mel:COC}
\textsc{R. Melrose} \emph{Calculus of conormal distributions on manifolds with corners} Intl. Math. Research Notices, No. 3  (1992), 51-61.

\bibitem[\textsc{Mel93}]{Mel:TAP}
\textsc{R. Melrose} \emph{The Atiyah-Patodi-Singer index theorem} Research Notes in Math., Vol. 4, A K Peters, Massachusetts (1993)

\bibitem[\textsc{Moo99}]{Moo:HKA}
\textsc{E.~A. Mooers}, \emph{Heat kernel asymptotics on manifolds with conic
  singularities}, J. Anal. Math. \textbf{78} (1999), 1--36. \MR{1714065
  (2000g:58039)}

\bibitem[\textsc{See69}]{See:TRO}
\textsc{R. Seeley} \emph{The resolvent of an elliptic boundary problem}, Amer. J. Math., No. 4  (1969), 889-920.


\bibitem[\textsc{Shu01}]{Shu:POS}
\textsc{M.~A. Shubin}, \emph{Pseudodifferential operators and spectral theory},
  second ed., Springer-Verlag, Berlin, 2001, Translated from the 1978 Russian
  original by Stig I. Andersson. \MR{1852334 (2002d:47073)}
  
\bibitem[\textsc{Ve15}]{Ver}
\textsc{B. Vertman} \emph{Regularized limit of determinants on discrete tori},
preprint arXiv:1502.04541 [math.SP], (2015)


\end{thebibliography}
\end{document}